\documentclass[11pt]{article}
\usepackage{geometry} 
\usepackage{amsmath,amsfonts,amsthm,amssymb, graphicx, setspace}
\usepackage{subfig}
\usepackage{color}

\newtheorem{theorem}{Theorem}

\newtheorem{assumption}{Assumption}

\theoremstyle{example}
\newtheorem{lemma}{Lemma}

\newtheorem{remark}{Remark}
\theoremstyle{remark}

\title{Three Asymptotic Regimes for Ranking and Selection with General Sample Distributions}
\author{Jing Dong, Yi Zhu}
\date{} 

\begin{document}
\maketitle

\begin{abstract}
In this paper, we study three asymptotic regimes that can be applied to ranking and selection (R\&S) problems with general sample distributions.
These asymptotic regimes are constructed by sending particular problem parameters (probability of incorrect selection, 
smallest difference in system performance that we deem worth detecting) to zero. We establish asymptotic validity and efficiency of the corresponding R\&S procedures in each regime. 
We also analyze the connection among different regimes and compare the pre-limit performances of corresponding algorithms. 
\end{abstract}

\section{Introduction}
\label{sec:intro}

Ranking and selection (R\&S) refers to the statistical procedure to select the simulated systems with the best performance (largest or smallest mean) among a finite number of alternatives with high probability. 
Most of the existing procedures are constructed under the assumption that the samples follow a Gaussian distribution or some other rather restricted class of distributions (e.g. sub-Gaussian, bounded support)
to gain control over the probability of correct selection. 
When these assumptions are violated, the desired performance can only be guaranteed in an asymptotic sense. 

Asymptotic analysis is achieved by sending the sample size to infinity. 
However, simply sending the sample size to infinity will not convey any meaningful information. 
By law of large numbers, it simply implies that the probability of correct selection will converge to one. 
To define the asymptotic regimes in a proper and meaningful way, we consider two parameters that
characterize the ``difficulty" of the problem: 1) the difference between the best system and the second best system, 
which we denote as $\delta$; 2) the probability of incorrect selection (PIS), which we denote as $\alpha$.
Under the indifference zone formulation, $\delta$ denotes the smallest difference in system performance
that we deem worth detecting. Thus, $\delta$ is also known as the indifference-zone parameter.
As either one of these parameters gets smaller, it requires more samples to achieve the desired performance.

In this paper, we study asymptotic regimes that can be applied for R\&S problems with general sample distributions.
The first limiting regime is called the central limit theorem regime, which is derived by sending $\delta$ to zero. 
The second limiting regime is called the large deviation regime, which is derived by sending $\alpha$ to zero. 
The third limiting regime is called the moderate deviation regime, which is derived by sending $(\alpha,\delta)$ to zero at an appropriate rate.
We present the theoretical foundation of each limiting regime and develop sequential stopping procedures for problems with unknown variance.

The central limit theorem regime has been applied in the R\&S literature, mostly under the the indifference zone formulation. 
\cite{Mukhopadhyay:1994} defined that an indifference zone procedure is asymptotically consistent if 
$\lim_{\delta \rightarrow 0} PIS \leq \alpha$.
\cite{Robbins:1968} propose a sequential stopping procedure for R\&S problems with unknown variance and show that their algorithm is asymptotically consistent. 
\cite{Kim:2005} develop a fully sequential selection procedure for steady-state simulation that is shown to be asymptotically consistent. 
This limiting regime has also been used to establish the asymptotic validity of sequential stopping procedures to construct fixed-width confidence intervals \cite{Glynn:1992}. 
The limit is achieved by sending the width of the confidence interval to zero. 
We will provide more details about this limiting regime in \S \ref{sec:clt}.

The large deviation regime has been applied in the ordinal optimization literature. 
The appealing fact is that while the width of the confidence interval decreases at rate $1/\sqrt{n}$ (due to the central limit theorem),  the PIS actually decays exponentially fast in $n$ (due to the large deviation theory) \cite{Dai:1996}. 
Results from this limiting regime, in particular, the large deviation rate function of the PIS, 
has been applied to find the optimal budget allocation rules, i.e. to minimize PIS under a fixed budget (see for example \cite{Glynn:2004}, \cite{Szechtman:2008} and \cite{Hunter:2013}). The large deviation type of upper bound on the probability of incorrect selection has also been applied in the multi-arm bandit literature \cite{Bubeck:2012}. Compared to the R\&S literature, the key performance measure for the multi-arm bandit literature is the regret, which measures the cumulative opportunity cost of not knowing the optimal system. Two of the well-known sampling strategies in this literature is the upper confidence bound strategy and the Thompson sampling. Both are shown to achieve an $O(log(n))$ regret bound, under the assumption that we have access to the large deviation rate function or an upper bound of the rate function in closed form. This assumption imposes constrains on the type of sample distributions we can work with. We will survey more details about this limiting regime in \S \ref{sec:ldp}. 

To the best of our knowledge, the moderate deviation regime studied in this paper has not been applied in the R\&S or the ordinal optimization literature,  though the moderate deviation theory is well-studied in the applied probability literature \cite{Dembo:1998}. 
As we shall explain in subsequent development (\S \ref{sec:mdp}), this asymptotic regime tends to strike a balance 
between the central limit theorem regime and the large deviation regime. 
 
\section{The Three Asymptotic Regimes} \label{sec:main}
To demonstrate the basic ideas, we restrict our discussion to the comparison 
between two systems. To formalize the asymptotic analysis,
we first define a suitable sequence of distributions. 
Let $X_1$ and $X_2^0$ be two random variables with the same mean $\mu_1$ and potentially different variance, $\sigma_1^2<\infty$ and $\sigma_2^2<\infty$, respectively. 
We define $X_2^{\delta}$, indexed by $\delta$, as a sequence of random variables with cumulative distribution function $F_2^{\delta}(x)=F_2^0(x+\delta)$, i.e. $X_2^{\delta} \,{\buildrel d \over =}\, X_2^{0}-\delta$. In particular,
$\mu_2^{\delta}=\mu_1-\delta$. 
We denote $X_{1,k}$, $k\geq 1$, as i.i.d. copies of $X_1$, and $X_{2,k}^\delta$, $k\geq 1$, as i.i.d. copies of $X_{2}^{\delta}$. 
Let
$$\bar X_1(n):=\frac{1}{n}\sum_{k=1}^{n} X_{1,k} ~~~\mbox{ and }~~~ \bar X_2^{\delta}(n):=\frac{1}{n}\sum_{k=1}^{n}X_{2,k}^{\delta}$$
denote the sample means of $X_1$ and $X_2^{\delta}$.  We also denote the sample variances as
$$S_1^2(n):=\frac{1}{n-1}\sum_{k=1}^{n}\left(X_{1,k}-\bar X_1(n)\right)^2 ~~~\mbox{ and }~~~ S_2^{\delta,2}(n):=\frac{1}{n-1}\sum_{k=1}^{n}\left(X_{2,k}^{\delta}-\bar X_2^{\delta}(n)\right)^2.$$
Our goal is to select the system with the largest mean value when comparing $X_1$ and $X_2^{\delta}$. In particular, if we draw $n_1$ samples from $X_1$ and $n_2$ samples from $X_2^{\delta}$, and select the system with the largest sample mean, then
$$PIS=P\left(\bar X_1(n_1)<\bar X_2^{\delta}(n_2)\right).$$

\begin{remark}
Other definitions of the sequence of random variables may also work. In general, 
we need to assume that the variances of $X_{2}^{\delta}$'s do not depend on $\delta$. 
\end{remark}

Recall that $\alpha$ denotes the required level of the PIS, 
and $\delta$ denotes the difference in mean between the two systems ($X_1$ and $X_2^{\delta}$).
We consider the following three asymptotic regimes.
i) Keep $\alpha$ fixed and send $\delta$ to zero;
ii) Keep $\delta$ fixed and send $\alpha$ to zero;
iii) Send both $\alpha$ and $\delta$ to zero at an appropriate rate.
We shall elaborate on each of these three regimes next.

\subsection{The Central Limit Theorem Regime}
\label{sec:clt}
In this limiting regime, we keep $\alpha$ fixed and send $\delta$ to zero. We start with the known variance case.
For fixed $q_1,q_2 > 1$ satisfying $1/q_1+1/q_2=1$, we set the required sample sizes as
$$n_i(\delta)=\frac{z_{\alpha}^2\sigma_i^2}{\delta^2} q_i ~~~\mbox{ for $i=1,2$.}$$
where $z_{\alpha}$ is the $\alpha$-th upper tail quantile of a standard normal distribution. 
We draw $n_i(\delta)$ samples from system i, and pick the system with the largest mean. The following theorem
establishes the asymptotic validity of this procedure.
\begin{theorem} \label{th:clt}
Under the assumption that $\sigma_i^2<\infty$ for $i=1,2$,
$$\lim_{\delta\rightarrow 0} P(\bar X_1(n_1(\delta))<\bar X_2^{\delta}(n_2(\delta)))=\alpha.$$
\end{theorem}
\begin{proof}[Proof of Theorem \ref{th:clt}]
We notice that
\begin{eqnarray*}
&&P\left(\bar X_1(n_1(\delta)) < \bar X_2^{\delta}(n_2(\delta))\right)\\
&=& P\left(\left(\bar X_1(n_1(\delta))-\mu_1\right)-\left((\bar X_2^{\delta}(n_2(\delta))+\delta)-\mu_1\right)<-\delta\right)\\
&=& P\left(\frac{\sqrt{n_1(\delta)}(\bar X_1(n_1(\delta))-\mu_1)}{\sigma_1\sqrt{q_1}}-\frac{\sqrt{n_2(\delta)}(\bar X_2(n_2(\delta))-\mu_1)}{\sigma_2\sqrt{q_2}}<-z_{\alpha} \right)\\
&\rightarrow& P(N(0,1)<-z_{\alpha}) ~~~\mbox{ as $\delta\rightarrow 0$}
\end{eqnarray*}
The convergence follows by Central Limit Theorem.
\end{proof}

We next introduce some possible choices of the parameter $q_i$'s.
i) If we are to minimize $n_1(\delta)+n_2(\delta)$ for each value of $\delta$,
then we set $q_1=1+\sigma_2/\sigma_1$ and $q_2=1+\sigma_1/\sigma_2$. In this case
$$n_i^*(\delta)=\frac{z_{\alpha}^2(\sigma_1+\sigma_2)}{\delta^2}\sigma_i ~~~\mbox{ for $i=1,2$.}$$
ii) If we want to draw equal amount of samples from both systems, 
then we set $q_i=(\sigma_1^2+\sigma_2^2)/\sigma_i^2$, for $i=1,2$. In this case 
$$n_1^{e}(\delta)=n_2^{e}(\delta)=z_{\alpha}^2(\sigma_1^2+\sigma_2^2)/\delta^2.$$ 
iii) If we want to run the simulation without taking into account the information of the other system, 
then we can set, for example, $q_1=q_2=2$. In this case
$$n_i^{in}(\delta)=\frac{2z_{\alpha}^2\sigma_i^2}{\delta^2} ~~~\mbox{ for $i=1,2$.}$$

When the variances are not known. We can apply the following sequential stopping procedure to decide the appropriate number of samples needed.    
In this paper, we will focus on the case of equal sample sizes only. 
We define the stopping time
$$\kappa(\delta):=\inf\left\{n \geq \delta^{-1}: z_{\alpha}^2\frac{S_1^2(n)+S_2^{\delta,2}(n)}{n}<\delta^2\right\},$$
where $\delta^{-1}$ is introduced to avoid early stopping. 
We keep sampling the two systems until the total sample variance over the sample size is smaller than 
$\delta^2/z_{\alpha}^2$, and then we pick the system with the largest sample mean.
The following theorem establishes the asymptotic validity of the sequential stopping procedure.

\begin{theorem} \label{th:clt_seq}
Under the assumption that $\sigma_i^2<\infty$ for $i=1,2$, 
$$\lim_{\delta\rightarrow 0} P\left(\bar X_1(\kappa(\delta))<\bar X_2^{\delta}(\kappa(\delta))\right)=\alpha.$$
\end{theorem}

\begin{proof}[Proof of Theorem \ref{th:clt_seq}]
For general $t \in \mathbb{R}^+$, define $\bar X_1(t):=\bar X_1(\lfloor t\rfloor)$, $\bar X_2^{\delta}(t):=\bar X_2^{\delta}(\lfloor t\rfloor)$,
$S_1^2(t):=S_1^2(\lfloor t\rfloor)$ and $S_2^{\delta,2}(t):=S_2^{\delta,2}(\lfloor t\rfloor)$.
We first notice that
\begin{eqnarray*} \label{eq:first}
\delta^2 \kappa(\delta) &=& \delta^2\inf\left\{n \geq \delta^{-1}: z_{\alpha}^2\frac{S_1^2(n)+S_2^{\delta,2}(n)}{n}<\delta^2\right\}  \\
&=&\inf\left\{t\geq \delta: z_{\alpha}^2\frac{S_1^2(t/\delta^2)+S_2^{\delta,2}(t/\delta^2)}{t}<1\right\} \\
&\,{\buildrel d \over =}\,&\inf\left\{t \geq \delta: z_{\alpha}^2\frac{S_1^2(t/\delta^2)+S_2^{0,2}(t/\delta^2)}{t}<1\right\}  \\
&=& \inf\left\{t \geq 0: z_{\alpha}^2\left(S_1^2((t+\delta)/\delta^2)+S_2^{0,2}((t+\delta)/\delta^2)\right)-(t+\delta)<0\right\}
\end{eqnarray*}
As $S_1^2((t+\delta)/\delta^2) \Rightarrow \sigma_1^2 I$ and $S_2^{0,2}((t+\delta)/\delta^2) \Rightarrow \sigma_2^2 I$ in $D[0,\infty)$ as $\delta \rightarrow 0$, where $I(t)=1$,
$$z_{\alpha}^2\left(S_1^2((t+\delta)/\delta^2)+S_2^{0,2}((t+\delta)/\delta^2)\right)-(t+\delta) \Rightarrow z_{\alpha}^2(\sigma_1^2+\sigma_2^2)-t \mbox{ in $D[0,\infty)$ as $\delta \rightarrow 0$}.$$
As $z_{\alpha}^2(\sigma_1^2+\sigma_2^2)-t$ is continuous and monotonically decreasing in $t$ \cite{Whitt:2002},
$$\delta^2 \kappa(\delta) \Rightarrow \inf\left\{t\geq 0: z_{\alpha}^2(\sigma_1^2+\sigma_2^2)-t<0\right\}=z_{\alpha}^2(\sigma_1^2+\sigma_2^2).$$
We next notice that 
\begin{eqnarray*}
&&P\left(\bar X_1(\kappa(\delta))<\bar X_2^{\delta}(\kappa(\delta))\right)\\
&=& P\left(\delta\kappa(\delta)(\bar X_1\left(\kappa(\delta)\right)-\bar X_2^{\delta}\left(\kappa(\delta)\right)-\delta)+\delta^2\kappa(\delta)<0\right)\\
&=& P\left(\delta\kappa(\delta)(\bar X_1\left(\kappa(\delta)\right)-\bar X_2^{0}\left(\kappa(\delta)\right))+\delta^2\kappa(\delta)<0\right)
\end{eqnarray*}
As
$\frac{t}{\delta}(\bar X_1(t/\delta^2)-\bar X_2^{0}(t/\delta^2))\Rightarrow \sqrt{\sigma_2^2+\sigma_2^2}B(t) \mbox{ in $D[0,\infty)$ as $\delta \rightarrow 0$}$,
using standard random time change and convergence together argument, we have
$$(\delta\kappa(\delta) t) (\bar X_1\left(\kappa(\delta) t\right)-\bar X_2^{0}\left(\kappa(\delta) t\right))+\delta^2\kappa(\delta)t \Rightarrow \sqrt{\sigma_1^2+\sigma_2^2}B\left(z_{\alpha}^2(\sigma_1^2+\sigma_2^2)t\right)+z_{\alpha}^2(\sigma_1^2+\sigma_2^2)t$$
in $D[0,\infty)$ as $\delta \rightarrow 0$.
Thus,
\begin{eqnarray*}
&&P\left(\delta\kappa(\delta)(\bar X_1\left(\kappa(\delta)\right)-\bar X_2^{0}\left(\kappa(\delta)\right))+\delta^2\kappa(\delta)<0\right)\\
&\rightarrow& P\left(\sqrt{\sigma_1^2+\sigma_2^2}B(z_{\alpha}^2(\sigma_1^2+\sigma_2^2))+z_{\alpha}^2(\sigma_1^2+\sigma_2^2)<0\right) \mbox{ as $\delta \rightarrow 0$}\\
&=& P(B(1) \leq -z_{\alpha})=\alpha.
\end{eqnarray*}
\end{proof}

\begin{remark}
Alternatively, we can also set
$$\kappa_i^{in}(\delta):=\inf\left\{n_i \geq \lfloor \delta^{-1} \rfloor: 2z_{\alpha}^2\frac{S_i^2(n_i)}{n_i}<\delta^2\right\},
~~~\mbox{ for $i=1,2$.}$$
Then we can show that
$\lim_{\delta\rightarrow 0} P\left(\bar X_1(\kappa_1^{in}(\delta))<\bar X_2^{\delta}(\kappa_2^{in}(\delta))\right)=\alpha$.
The separation of the simulation of the two systems may become handy in parallelization.
\end{remark}

\subsection{The Large Deviation Regime}
\label{sec:ldp}
In this limiting regime, we keep $\delta$ fixed and send $\alpha$ to zero. We impose light tail assumptions on the sample distribution.
\begin{assumption} \label{ass:ldp}
There exists $\theta>0$ such that $E[\exp(\theta X_1)]<\infty$ and $E[\exp(\theta X_2^0)]<\infty$.
\end{assumption}
We next introduce a few notations. Let $\psi_1(\theta):=\log E[\exp(\theta X_1)]$ and $\psi_2(\theta):=\log E[\exp(\theta X_2^\delta)]$, i.e. the log moment generating functions. We also write 
$I_i(a):=\sup_{\theta}\{\theta a - \psi_i(\theta)\},$
which is known as the Fenchel-Legendre transformation of $\psi_i$, for $i=1,2$.
Let $\mathcal{D}_i=\{\theta \in \mathbb{R}: \psi_i(\theta)<\infty\}$ and $\mathcal{S}_i=\{\psi_i^{\prime}(\theta): \theta \in \mathcal{D}\}$.
It is well-known that $I_i$ is strictly convex and $C^{\infty}$ for $a \in \mathcal{S}_i$.  We also make the following assumption on the sample distribution
\begin{assumption} \label{ass:int}
The interval $[\mu_1,\mu_1+\delta]\subset \mathcal{S}_1^o \cap \mathcal{S}_2^o$.
\end{assumption}

We assume that $I_i$'s are known. For fixed $p_1, p_2 >0$ with $p_1+p_2=1$. We can interpret $p_i$ as the proportion of sampling budget
allocated to system $i$, $i=1,2$.
We denote
$G(p_1,p_2)=\min_{b\in(\mu_1,\mu_1+\delta)} \{p_1I_1(b)+p_2I_2(b)\}$. 
Then set the sample size
$$\tilde n_i(\alpha)=\frac{\log(1/\alpha)}{G(p_1,p_2)}p_i ~~~ \mbox{ for $i=1,2$. }$$ 
We draw $\tilde n_i(\alpha)$ from system $i$, $i=1,2$, and pick the system with largest sample mean. The following theorem establishes the asymptotic
validity, in a logarithmic sense, of the procedure.
\begin{theorem} \label{th:ldp}
Under Assumption \ref{ass:ldp} \& \ref{ass:int},
$$\lim_{\alpha\rightarrow 0} \frac{\log P(\bar X_1(\tilde n_1(\alpha))<\bar X_2^{\delta}(\tilde n_2(\alpha))}{\log(\alpha)}=1.$$ 
\end{theorem}
The proof of the theorem follows from the same line of analysis as in \cite{Glynn:2004}. We shall only provide an outline here.
Let $\tilde n(\alpha)=\log(1/\alpha)$. Then $\tilde n_i(\alpha)=\tilde n(\alpha)p_i/G(p_1,p_2)$ for $i=1,2$.
We also denote 
$Z(n)=\left(\bar X(np_1/G(p_1,p_2)),\bar X(np_2/G(p_1,p_2))\right)$. Then the rate function of $\{Z(n): n\geq 0\}$, is (Lemma 1 in \cite{Glynn:2004})
$$I(x_1,x_2)=\frac{p_1}{G(p_1,p_2)}I_1(x_1)+\frac{p_2}{G(p_1,p_2)}I_2(x_2).$$
Then we have,
\begin{eqnarray*}
&&\lim_{n\rightarrow \infty}\frac{1}{n}\log P\left(\bar X(np_1/G(p_1,p_2))<\bar X(np_2/G(p_1,p_2))\right)\\
&=&-\inf_{b\in(\mu_1,\mu_1-\delta)}\left(\frac{p_1}{G(p_1,p_2)} I_1(b)+\frac{p_2}{G(p_1,p_2)} I_2(b)\right)=-1.
\end{eqnarray*}
Thus,
$\lim_{\alpha \rightarrow \infty}\frac{1}{\tilde n(\alpha)}\log P\left(\bar X(\tilde n_1(\alpha))<\bar X(\tilde n_2(\alpha))\right)=-1.$

We next provide some special choices of $p_i$'s. i) If we want to minimize the sampling cost, 
then we pick $(p_1,p_2)$
that solves
$$\min_{p_1,p_2} (p_1+p_2)/G(p_1,p_2) ~~~\mbox{ s.t. }~ p_1+p_2=1, p_1>0, p_2>0.$$
ii) If we want to draw equal amount of samples from the two systems, then we set $p_1=p_2=1/2$. In this case
$$\tilde n_i^{e}(\alpha)=\frac{\log(1/\alpha)}{2G(1/2,1/2)} \mbox{ for $i=1,2$.}$$
iii) It is also possible to draw samples from each system without taking into account the information of the other system.
For example, we can pick any $b \in (\mu_1,\mu_1+\delta)$ and set
$$\tilde n_i^{in}(\alpha)=\frac{\log(1/\alpha)}{I_i(b)} \mbox{ for $i=1,2$.}$$
However, in this case we ``overshoot" the PIS. In particular, following the proof of Theorem \ref{th:ldp}, 
it is easy to check that
$$\lim_{\alpha\rightarrow 0} \frac{\log P(\bar X_1(\tilde n_1^{in}(\alpha))<\bar X_2^{\delta}(\tilde n_2^{in}(\alpha))}{\log(\alpha)}<1.$$ 
 
In applications, the assumption that $I_i(\cdot)$'s are known is rather restrictive. When $I_i(\cdot)$'s are not known, estimating
this function would in general be a more difficult task than estimating the means. Recently, \cite{Glynn:2015} conduct an extensive analysis of this issue.
 
\subsection{The Moderate Deviation Regime}
\label{sec:mdp}

In this regime, we send both $\alpha$ and $\delta$ to zero at an appropriate rate. In particular, we consider a sequence of $(\alpha_k,\delta_k)$'s, indexed by $k\in \mathbb{N}$, satisfying that $\alpha_k \rightarrow 0$ and $\delta_k \rightarrow 0$ as $k\rightarrow \infty$, 
and
$$\log(1/\alpha_k)\delta_k^{(1-2\beta)/\beta}=L,$$
for some $\beta\in(1/3,1/2)$ and $L>0$, independent of $k$.

We start by assuming that the variances are known. 
In this case, for fixed $p_1, p_2>0$ with $p_1+p_2=1$, we set 
$$\hat n_i(k)=\frac{\log(1/\alpha_k)}{\delta_k^2\hat G(p_1,p_2)}p_i, ~~~ \mbox{ for $i=1,2$}.$$
where $\hat G(p_1,p_2)=p_1p_2/(2(\sigma_1^2 p_2+\sigma_2^2p_1))$.
For $\delta_k=\delta$, $\alpha_k=\alpha$, we draw $\hat n_i(k)$ samples from system $i$, and then choose the system with the largest sample mean. The following theorem establishes the asymptotic validity, in a logarithmic sense, of the procedure.
\begin{theorem} \label{th:mdp}
Under Assumption \ref{ass:ldp},
$$\lim_{k\rightarrow \infty}\frac{\log P\left(\bar X_1(\hat n_1(k))<\bar X_2^{\delta_k}(\hat n_2(k))\right)}{\log(\alpha_k)}=1 .$$
\end{theorem}
\begin{proof}[Proof of Theorem \ref{th:mdp}]
Let 
$\hat n(k)=\log(1/\alpha_k)\delta_k^{-2}$, $q_i=p_i/\hat G(p_1,p_2)$. 
\begin{eqnarray*}
&&P\left(\bar X_1(\hat n_1(k))<\bar X_2^{\delta_k}(\hat n_2(k))\right)\\
&=&P\left((\bar X_1(\hat n(k)q_1)-\mu_1)-(\bar X_2^{0}(\hat n(k)q_2)-\mu_1)<-\delta_k\right)\\
&=& P\left((\bar X_1(\hat n(k)q_1)-\mu_1)-(\bar X_2^{0}(\hat n(k)q_2)-\mu_1)<-L^{\beta} \hat n(k)^{-\beta}\right)\\
&=& P\left(\hat n(k)^{\beta}(\bar X_1(\hat n(k)q_1)-\mu_1)-\hat n(k)^{\beta}(\bar X_2^{0}(\hat n(k)q_2)-\mu_1)<-L^{\beta}\right).
\end{eqnarray*}
Let $Z_n=(n^{\beta}(\bar X_1(nq_1)-\mu_1), n^{\beta}(\bar X_2^0(nq_2)-\mu_1))$, then
$$\lim_{n\rightarrow\infty}\frac{1}{n^{1-2\beta}}E\left[\exp(n^{1-2\beta}\theta Z_n)\right]
=q_1\frac{1}{2}\left(\frac{\theta_1}{q_1}\right)^2\sigma_1^2+q_2\frac{1}{2}\left(\frac{\theta_2}{q_2}\right)^2\sigma_2^2$$
By Gartner-Ellis theorem, $Z_n$ satisfies a LDP with rate $n^{1-2\beta}$ and rate function
$$\hat I(x_1,x_2)=q_1x_1^2/(2\sigma_1^2)+q_2x_2^2/(2\sigma_2^2).$$
In particular,
$$\lim_{n\rightarrow\infty}\frac{1}{n^{1-2\beta}}\log P(Z_n(1)-Z_n(2)<-L^{\beta})=-\inf_{x_1,x_2, x_1-x_2<-L^{\beta}}\hat I(x_1,x_2)=-L^{2\beta}.$$
As $L^{2\beta}\hat n(k)^{1-2\beta}=\log(1/\alpha_k)$, we have
$$\lim_{k\rightarrow\infty}\frac{1}{\log(\alpha_k)}\log P\left(\hat n(k)^{\beta}(\bar X_1(\hat n(k)q_1)-\mu_1)-\hat n(k)^{\beta}(\bar X_2^{0}(\hat n(k)q_2)-\mu_1)<-L^{\beta}\right)=1.$$
\end{proof}

We next provide some special choices of $p_i$'s. i) If we want to minimize the total sampling cost $\hat n_1(k)+\hat n_2(k)$ for each $k$, 
then we pick $p_i=\sigma_i/(\sigma_1+\sigma_2)$ for $i=1,2$. In this case
$$\hat n_i^*(k)=\frac{\log(1/\alpha_k)(\sigma_1+\sigma_2)\sigma_i}{\delta_k^2} ~~~\mbox{ for $i=1,2$.}$$
ii) When $p_1=p_2=1/2$, we draw equal amount of samples from the two systems. In this case
$$\hat n_1^e(k)=\hat n_2^e(k)=\frac{2\log(1/\alpha_k)(\sigma_1^2+\sigma_2^2)}{\delta_k^2}$$
iii) It is also possible to draw samples from each system without taking into account the information of the other system. 
In particular, when $p_i=\sigma_i^2/(\sigma_1^2+\sigma_2^2)$, 
$$\hat n_i^{in}(k)=\frac{4\log(1/\alpha_k) \sigma_i^2}{\delta_k^2} ~~~\mbox{for $i=1,2$.}$$

When the variances are not known a priori, we introduce a sequential stopping procedure. In this paper, we shall focus on the case of equal sample sizes only. 
We also impose the following assumption on sample distribution (mainly for technical reasons).

\begin{assumption} \label{ass:mdp}
There exist $\theta>0$ such that $E[\exp(\theta (X_1-\mu_1)^2)]<\infty$ and $E[\exp(\theta (X_2^0-\mu_1)^2)]<\infty$.
\end{assumption}

We define the stopping time
$$N_k:=\inf\left\{n \geq \delta_k^{-1}: 2\frac{S_1^2(n)+S_2^{\delta_k,2}(n)}{n} < \delta_{k}^2\log(1/\alpha_k)^{-1}\right\}.$$
We keep sampling the two systems until the total sample variance over the sample size is smaller than 
$\delta_k^2/(2\log(1/\alpha_k))$, and then we pick the system with the largest sample mean.
The following theorem establishes the asymptotic validity of the sequential stopping procedure. 

\begin{theorem} \label{th:mdp_seq}
Under Assumption \ref{ass:mdp},
$$\lim_{k\rightarrow \infty}\frac{\log P\left(\bar X_1(N_k)<\bar X_2^{\delta_k}(N_k)\right)}{\log(\alpha_k)}=1 .$$
\end{theorem}

\begin{remark}
Alternatively, we can also set
$$N_i^{in}(k):=\inf\left\{n_i \geq \delta_k^{-1}: \frac{4S_i^2(n_i)}{n_i}<\delta_{k}^2\log(1/\alpha_k)^{-1}\right\},
~~~\mbox{ for $i=1,2$.}$$
Then following the same line of analysis as in the Proof of Theorem \ref{th:mdp_seq}, we have
$$\lim_{k\rightarrow \infty} \frac{1}{\log(\alpha_k)}P\left(\bar X_1(N_1^{in}(k))<\bar X_2^{\delta_k}(N_2^{in}(k))\right)=1.$$
The separation of the simulation of the two systems may become handy in parallelization.
\end{remark}

\subsubsection{Proof of Theorem \ref{th:mdp_seq}}
We first notice that as $\log(1/\alpha_k)\delta_k^{(1-2\beta)/\beta}=L$,
$$N_k=\inf\left\{n \geq \delta_k^{-1}: 2L\frac{S_1^2(n)+S_2^{\delta_k,2}(n)}{n} \leq \delta_k^{1/\beta}\right\}.$$
As the distribution of $S_2^{\delta_k,2}(n)$ does not depend on $\delta_k$. We shall drop the superscription $\delta_k$
when there is no confusion. Let $S_i^2(t):=S_i^2(\lfloor t \rfloor)$ for $t\geq 0$.
As $S_i^2(n) \Rightarrow \sigma_i^2$ as $n \rightarrow \infty$, by Continuous Mapping Theorem 
$$\delta_k^{1/\beta}N_k=\inf\left\{t \geq \delta^{1/\beta-1}: 2L\frac{S_1^2(t\delta_k^{-1/\beta})+S_2^{2}(t\delta_k^{-1/\beta})}{t}<1\right\}
\Rightarrow 2L(\sigma_1^2+\sigma_2^2).$$
We next establish an upper bound for 
$P\left(|\delta_k^{1/\beta}N_k -2L(\sigma_1^2+\sigma_2^2)|>\epsilon\right)$ for any $\epsilon>0$ small enough.
\begin{eqnarray*}
&&\frac{1}{n}\log E\left[\exp(n\theta S_i^2(n)\right]\\
&=&\log E\left[\exp(\theta(X_i-\mu_i)^2)\right]-\frac{1}{n}\log E\left[\exp(-\theta n(\bar X_i(n)-\mu_i)^2)\right]\\
&\rightarrow& \log E\left[\exp(\theta(X_i-\mu_i)^2)\right] ~~~\mbox{ as $n\rightarrow \infty$,}
\end{eqnarray*}
for $i=1,2$. Thus,
$$\lim_{n\rightarrow \infty}\frac{1}{n}\log E\left[\exp(n\theta (S_1^2(n)+S_2^2(n)))\right] = \tilde\psi_1(\theta)+ \tilde\psi_2(\theta),$$
where $\tilde\psi_i(\theta)=\log E[\exp(\theta (X_i-\mu_i)^2)]$.
By Gartner-Ellis Theorem, $S_1^2(n)+S_2^2(n)$ satisfies a LDP with rate function
$I(a)=\sup\left\{\theta a - \left(\tilde\psi_1(\theta)+ \tilde\psi_2(\theta)\right)\right\}$.
Then we have
\begin{eqnarray*}
&&P\left(\delta_k^{1/\beta}N_k>2L(\sigma_1^2+\sigma_2^2)+\epsilon\right)\\
&\leq& P\left(S_1^2\left((2L(\sigma_1^2+\sigma_2^2)+\epsilon)\delta^{-1/\beta}\right)+S_2^2\left((2L(\sigma_1^2+\sigma_2^2)+\epsilon)\delta^{-1/\beta}\right)>\sigma_1^2+\sigma_2^2+\epsilon\right)\\
&\leq& \exp(-(2L(\sigma_1^2+\sigma_2^2)+\epsilon)\delta_k^{-1/\beta}I(\epsilon)+o(\delta_k^{-1/\beta})),
\end{eqnarray*} 
and
\begin{eqnarray*}
&&P\left(\delta_k^{1/\beta}N_k<2L(\sigma_1^2+\sigma_2^2)-\epsilon\right)\\
&=&P\left( \exists t \in(\delta_k^{1/\beta-1},  2L(\sigma_1^2+\sigma_2^2)-\epsilon) \mbox{ s.t. } 2L\left(S_1^2\left(t\delta_k^{-1/\beta}\right)+S_2^2\left(t\delta_k^{-1/\beta}\right)\right)<t\right)\\
&\leq& \sum_{n=\lfloor \delta_k^{-1} \rfloor}^{\lfloor (2L(\sigma_1^2+\sigma_2^2)-\epsilon)\delta_k^{-1/\beta}\rfloor}P\left(S_1^2\left(n\right)+S_2^2\left(n\right)<(\sigma_1^2+\sigma_2^2)-\epsilon\right)\\
&\leq& \left(2L(\sigma_1^2+\sigma_2^2)-\epsilon)\delta_k^{-1/\beta}\right)\exp(-\delta_k^{-1}I(\epsilon)+o(\delta_k^{-1}))
\end{eqnarray*} 
We denote $B_{\epsilon}(k):=\{t: |\delta_k^{1/\beta} t - 2L(\sigma_1^2+\sigma_2^2)|<\epsilon\}$ for any $\epsilon>0$. 
Then $P(N_k \in B_{\epsilon}(k)) \leq \exp(-\delta_k^{-1}I(\epsilon)+o(\delta_k^{-1})).$
Let $N_k^*=2L(\sigma_1^2+\sigma_2^2)\delta_k^{-1/\beta}$ and $a_{\epsilon}(n)=n/N_k^*$.
Then for any $n \in B_{\epsilon}(k)$.
$$\frac{2L(\sigma_1^2+\sigma_2^2)-\epsilon}{2L(\sigma_1^2+\sigma_2^2)}\leq a_{\epsilon}(n)\leq \frac{2L(\sigma_1^2+\sigma_2^2)-\epsilon}{2L(\sigma_1^2+\sigma_2^2)}$$ 
We also notice that
\begin{eqnarray*}
&&P(\bar X_1(N_k)<\bar X_2^{\delta_k}(N_k))\\
&=&P(\bar X_1(N_k)<\bar X_2^{\delta_k}(N_k)|N_k \in B_{\epsilon}(k))P(N_k \in B_{\epsilon}(k))\\
&&+P(\bar X_1(N_k)<\bar X_2^{\delta_k}(N_k)|N_k \not\in B_{\epsilon}(k))P(N_k \not\in B_{\epsilon}(k)).
\end{eqnarray*}
Thus,
\begin{eqnarray*}
&&P(\bar X_1(N_k)<\bar X_2^{\delta_k}(N_k))\\ 
&\leq& P(\bar X_1(N_k)<\bar X_2^{\delta_k}(N_k)|N_k \in B_{\epsilon}(k))+P(N_k \not\in B_{\epsilon}(k))\\
&\leq& \sup_{n\in B_{\epsilon}(k)} P(\bar X_1(n) < \bar X_2^{\delta_k}(n))+P(N_k \not\in B_{\epsilon}(k))\\
&=& \sup_{n\in B_{\epsilon}(k)} P\left(\bar X_1(N_k^* a_\epsilon(n)) < \bar X_2^{\delta_k}(N_k^* a_{\epsilon}(n))\right)+P(N_k \not\in B_{\epsilon}(k))\\
&\leq& \exp\left(-\delta_k^{-(1/\beta-2)}L\left(\frac{L(\sigma_1^2+\sigma_2^2)-\epsilon}{L(\sigma_1^2+\sigma_2^2)}\right)+ o\left(\delta_k^{-(1/\beta-2)}\right)\right)(1+o(1)),
\end{eqnarray*}
where the second inequality use the fact that $S_i^2(n)$ is independent of $\bar X_i(n)$, 
the third inequality follows from the proof of Theorem \ref{th:mdp} and the fact that $0<1/\beta-2<1$.
Similarly,
\begin{eqnarray*}
&&P(\bar X_1(N_k)<\bar X_2^{\delta_k}(N_k)) \\
&\geq& P(\bar X_1(N_k)<\bar X_2^{\delta_k}(N_k)|N_k \in B_{\epsilon}(k))P(N_k \in B_{\epsilon}(k))\\
&\geq& \inf_{n\in B_{\epsilon}(k)} P(\bar X_1(n) < \bar X_2^{\delta_k}(n))P(N_k\in B_{\epsilon}(k))\\
&\geq& \exp\left(-\delta_k^{-(1/\beta-2)}L\left(\frac{L(\sigma_1^2+\sigma_2^2)+\epsilon}{L(\sigma_1^2+\sigma_2^2)}\right)+ o\left(\delta_k^{-(1/\beta-2)}\right)\right)(1+o(1)).
\end{eqnarray*}
As $L\delta_k^{-(1/\beta-2)}=\log(\alpha_k)$ and $\epsilon$ can be arbitrarily small, we have 
$$\lim_{k\rightarrow \infty} \log P\left(\bar X_1(N_k)<\bar X_2^{\delta_k}(N_k)\right)/\log(\alpha_k)=1.$$

\section{Comparison of the Three Asymptotic Regimes} \label{sec:compare}
The three asymptotic regimes are closely related to each other. Figure \ref{fig:demo} provide an overview of their relationships. We've established the three solid arrows in the figure in \S \ref{sec:main}. We next show the two dotted arrows for some special cases (Lemma \ref{lm:clt} and Lemma \ref{lm:ldp}).
\begin{figure}[htb]
{
\centering
\includegraphics[width=0.5\textwidth]{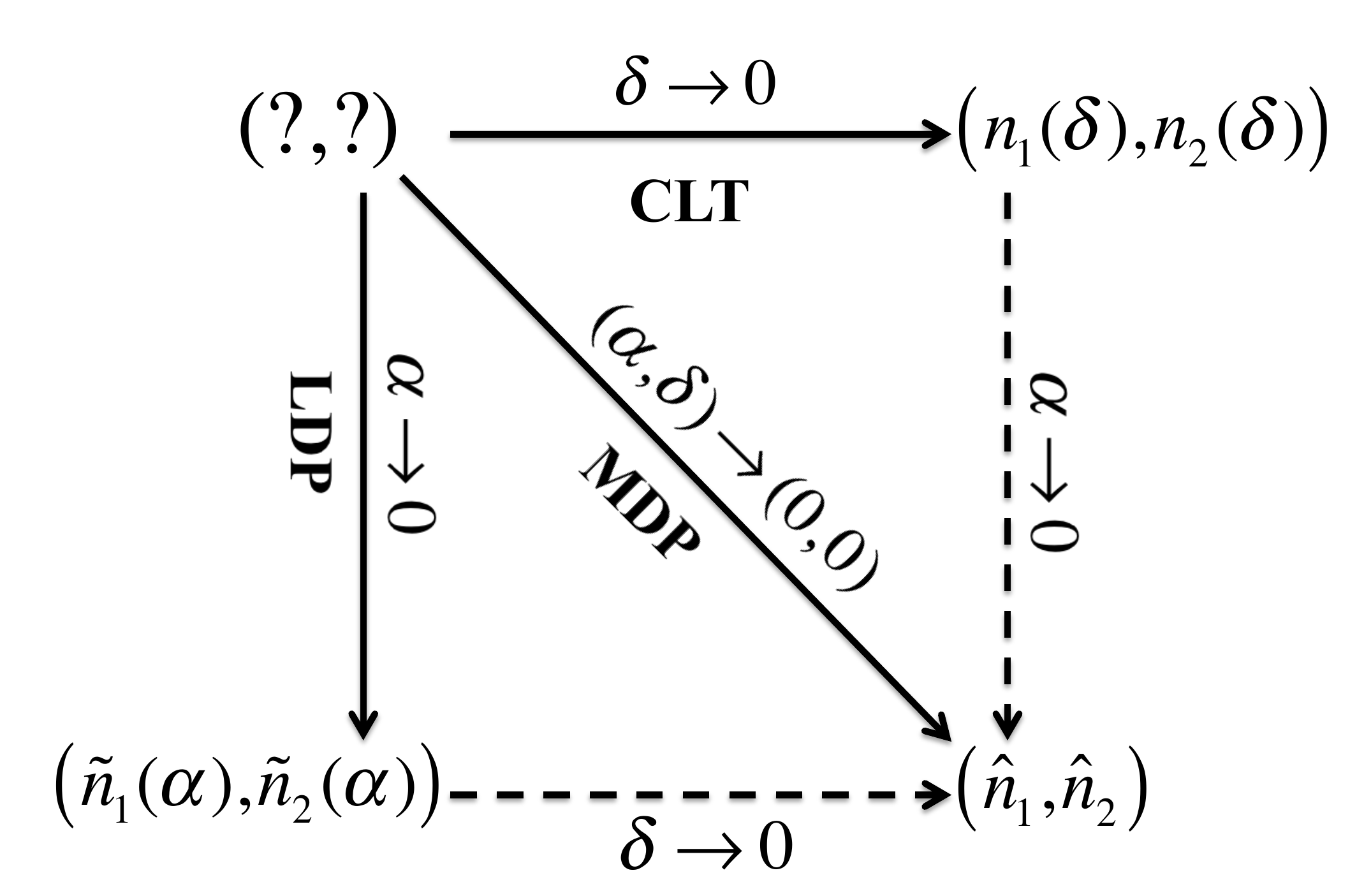}
\caption{Relationship of the three asymptotic regimes. (CLT: Central Limit Theorem, LDP: Large Deviation Principle, MDP: Moderate Deviation Principle)\label{fig:demo}}
}
\end{figure}

\begin{lemma}\label{lm:clt}
The $(1-\alpha)$-th quantile of standard Normal distribution, $z_{\alpha}$, satisfies,
$$\frac{z_{\alpha}^2}{2\log(1/\alpha)}\rightarrow 1 ~~\mbox{ as $\alpha\rightarrow 0$}.$$
\end{lemma}
Lemma \ref{lm:clt} implies for $\alpha$ small enough, and $\alpha_k=\alpha$, we have $n_i^*(\delta_k)\approx \hat n_i^*(k)$, $n_i^{e}(\alpha)\approx \hat n_i^{e}(k)$, and $n^{in}(\alpha_k)\approx \hat n^{in}(k)$.

For the large deviation regime, when we impose equal sample sizes from the two systems, then we have
$P(\bar X_1(n)<\bar X_2^{\delta}(n))=P(\bar X_2^0(n)-\bar X_1(n)>\delta)$. Let $\psi(\theta):=\log E[\exp(X_2^0-X_1)]$.
We also define
$\mathcal{D}:=\{\theta \in \mathbb{R}: \psi(\theta)<\infty\}$ and $\mathcal{S}:=\{\psi^{\prime}(\theta): \theta \in \mathcal{D}\}$.
If we write $G_e(\delta)=\sup_{\theta}\{\theta\delta-\psi(\theta)\}$, then $\tilde n_1^e(\alpha)=\tilde n_2^e(\alpha)=\log(1/\alpha)/G_e(\delta)$.
\begin{lemma} \label{lm:ldp}
For $\delta \in \mathcal{S}$,
$$G_e(\delta)=\frac{1}{2}\frac{1}{\sigma_1^2+\sigma_2^2}\delta^2-\frac{1}{6}\frac{E[(X_2^{0}-X_1)^3]}{\left(\sigma_1^2+\sigma_2^2\right)^3}\delta^3+ O (\delta^4).$$
\end{lemma}
Lemma \ref{lm:ldp} implies that for $\delta$ small, $\delta_k=\delta$, $\tilde n_i^e(\alpha_k)\approx \hat n_i^e(k)$.

\subsection{Pre-Limit Performance}
We next provide some comments about the pre-limit performance, i.e. for fixed $\alpha$ and $\delta$. For simplicity of exposition, we shall restrict our discussion to the case of equal sample sizes.

In the Central Limit Theorem regime, when $X_1$ and $X_2^{\delta}$ are Gaussian random variables, we have 
$P\left(\bar X_1(n_1^e(\delta))<\bar X_2^{\delta}(n_2^e(\delta))\right)=\alpha$. 
In general, the performance of ``$n_i^e(\delta)$" depends on the ``rate" of convergence of the central limit theorem.  
Assume that $E[(X_2^0-X_1)^4]<\infty$ and $X_2^0-X_1$ is non-lattice.
Let 
$$S_{kew}:=E\left[\left((X_2^{0}-X_1-0)/\sqrt{\sigma_1^2+\sigma_2^2}\right)^3\right]$$ 
denote the skewness of $X_2^{0}-X_1$. Skewness is a measure of asymmetry of a random variable about its mean. We also write
$$K_{ur}:=E\left[\left((X_2^{0}-X_1-0)/\sqrt{\sigma_1^2+\sigma_2^2}\right)^4\right]$$ 
as the Kurtosis of $X_2^{0}-X_1$. Kurtosis measures the heaviness of the tail of a random variable.
Let $\bar\Phi(\cdot)$ and $\phi(\cdot)$ denote the tail cumulative distribution function and the probability density function of a standard normal distribution.
Using the Edgeworth expansion \cite{Shao:1995}, we can show that
\begin{eqnarray*}
&&P\left(\bar X_1(n_1^e(\delta))<\bar X_2^{\delta}(n_1^e(\delta))\right)-\alpha\\
&=&\frac{\phi(z_{\alpha})}{z_{\alpha}}\left\{\frac{S_{kew}(z_\alpha^2-1)}{6\sqrt{\sigma_1^2+\sigma_2^2}}\delta+\left(\frac{(K_{ur}-3)(z_{\alpha}^2-3)}{24(\sigma_1^2+\sigma_2^2)}+\frac{S_{kew}^2(z_{\alpha}^4-10z_{\alpha}^2+15)}{72(\sigma_1^2+\sigma_2^2)}\right)\delta^2\right\}+o(\delta^2)
\end{eqnarray*}

We make the following observations.
a) $\phi(z_{\alpha})/z_{\alpha} > \alpha$ with $\lim_{\alpha\rightarrow 0}\frac{\phi(z_{\alpha})/z_{\alpha}}{\alpha}=1$.
b) For distribution with large skewness or Kurtosis, the pre-limit PIS may be quite different from $\alpha$.  
A common practice to reduce the approximation error (improve the rate of convergence) is to use the Cornish-Fisher expansion \cite{Shao:1995} to refine the scaling parameter $z_{\alpha}$, but this would require us to know higher moments of the sample distributions. 

For the large deviation regime, we first notice that, using Chernoff's bound, 
$$P\left(\bar X_1(\tilde n_1^e(\alpha))<\bar X_2^{\delta}(\tilde n_2^e(\alpha)\right) \leq \exp(-\tilde n_1^e(\alpha)G(\delta))=\alpha.$$
To quantify how much smaller $P\left(\bar X_1(\tilde n_1(\alpha))<\bar X_2^{\delta}(\tilde n_1(\alpha))\right)$ is, compared to $\alpha$, we 
refer to a refinement of the large deviation asymptotic approximation due to \cite{Bahadur:1960}.
Assume that $\psi(\theta)$ is steep on the right and $X_1-X_2^0$ is non-lattice. We denote $\theta(\delta):=\arg\min\{\theta \delta - \psi(\theta)\}$. 
Then we can show that
$$P\left(\bar X_1(\tilde n_1^e(\alpha))<\bar X_2^{\delta}(\tilde n_2^e(\alpha)\right)=
\frac{\alpha}{\sqrt{\log(1/\alpha)}} \frac{\sqrt{G(\delta)}}{\sqrt{2\pi \psi^{\prime\prime}(\delta)}\theta(\delta)}\left(1+O\left(\frac{\sqrt{G(\delta)}}{\sqrt{\log(1/\alpha)}}\right)\right).$$
As
$\lim_{\delta \rightarrow 0}\sqrt{G(\delta)}/(\sqrt{\psi^{\prime\prime}(\delta)}\theta(\delta))=1/\sqrt{2}$,
when $\delta$ is small enough, $P\left(\bar X_1(\tilde n_1(\alpha))<\bar X_2^{\delta}(\tilde n_1(\alpha)\right)$ decays at rate $\alpha /\sqrt{\log(1/\alpha)}$ approximately, which is slightly faster than $\alpha$. Therefore, sampling rules derived from the large deviation regime 
provide a guarantee on PIS but tend to over-sample in practical examples.

For the moderate deviation regime, we first notice that
$z_{\alpha}^2<2\log(1/\alpha)$ for fixed value of $\alpha$.
Thus, we tend to sample more than ``needed" when the central limit theorem regime works well. 
However, this also provides us with a safety buffer when the central limit theorem regime doesn't work well.

\subsection{Numerical Comparison}
The following numerical experiments illustrate the pre-limit performance of the sampling rules derived from the three asymptotic regimes (Table \ref{tab:num1} \& \ref{tab:num2}). The probability of incorrect selection are calculated based $10^6$ independent experiments.
In Table \ref{tab:num1}, we assume both systems have Exponential sample distributions.  
We observe that in this case, the CLT regime sampling rule achieves the desired probability of incorrect selection, but the other two regimes overshoot the probability of incorrect selection, i.e. $PIS\ll 0.05$. Table \ref{tab:num2} illustrate an extreme example, where system 1 has constant output while system 2 has Bernoulli sample distribution with very small probability of success (highly skewed). There the CLT regime doesn't achieve the desired probability of incorrect selection while the LD regime over-samples.
We would also like to point out that among the three regimes, the LD regime is the only one that is guaranteed to have $PIS<\alpha$ regardless of the sample distributions. 
\begin{table}[htb]
\centering
\caption{Simulation experiments for Exponential samples ($\mu_1=1$, $\mu_2^{\delta}=0.9091$, $\alpha=0.05$).\label{tab:num1}}
\begin{tabular}{rll}
\hline
Regime & $n$ & Probability of Incorrect Selection\\ \hline
CLT Regime & $598$ & $0.0497 \pm0.0002$\\
LD Regime & $1320$ & $0.0072 \pm 0.0001$\\
MD Regime & $1325$ & $0.0071 \pm 0.0001$ \\
\hline
\end{tabular}
\end{table}

\begin{table}[htb]
\centering
\caption{Simulation experiments for constant and Bernoulli samples ($\mu_1=0.008$, $\mu_2^{\delta}=0.001$, $\alpha=0.01$).\label{tab:num2}}
\begin{tabular}{rll}
\hline
Regime & $n$ & Probability of Incorrect Selection\\ \hline
CLT Regime & $111$ & $0.1057 \pm0.0003$\\
LD Regime & $477$ & $0.0015 \pm 0.00003$\\
MD Regime & $188$ & $0.0156 \pm 0.0001$ \\
\hline
\end{tabular}
\end{table}



\appendix

\section{Proofs} \label{app:quadratic}

\begin{proof}[Proof of Lemma \ref{lm:clt}]
We first notice that
$\left(\frac{1}{x}-\frac{1}{x^3}\right)\phi(x) \leq \bar \Phi(x)\leq \frac{1}{x}\phi(x).$
As $z_{\alpha} \rightarrow \infty$ as $\alpha \rightarrow 0$, then
\begin{eqnarray*}
\lim_{\alpha \rightarrow 0} \frac{z_{\alpha}^2}{2\log(1/\alpha)}=\lim_{\alpha \rightarrow 0} \frac{z_{\alpha}^2}{-2\log(\bar\Phi(z_{\alpha}))}
&\geq& \lim_{\alpha \rightarrow 0} \frac{z_{\alpha}^2}{-2\log(\phi(z_{\alpha}))-2\log(1/z_{\alpha})}\\
&=&\lim_{\alpha \rightarrow 0}\frac{z_{\alpha}^2}{z_{\alpha}^2+2\log(z_{\alpha})}=1
\end{eqnarray*}
Similarly,
$$\lim_{\alpha \rightarrow 0} \frac{z_{\alpha}^2}{2\log(1/\alpha)}\leq \lim_{\alpha \rightarrow 0} \frac{z_{\alpha}^2}{-2\log(\phi(z_\alpha))-\log(1/z_{\alpha}-1/z_{\alpha}^3)}=1.$$
Thus, 
$$\lim_{\alpha \rightarrow 0} \frac{z_{\alpha}^2}{2\log(1/\alpha)}=1.$$
\end{proof}

\begin{proof}[Proof of Lemma \ref{lm:ldp}]
Applying Taylor expansion, we have for $\delta \in \mathcal{S}$
$$G_e(\delta)=G_e(0)+G_e^{\prime}(0)\delta + \frac{1}{2}G_e^{\prime\prime}(0)\delta^2+\frac{1}{6}G_e^{\prime\prime\prime}(0)\delta^3+O (\delta^4)$$
Let $\theta(\delta):=\arg\min\{\theta\delta-\psi(\theta)\}$. Then $G_e(\delta)=\theta(\delta)\delta-\psi(\theta(\delta))$ and
\begin{eqnarray*}
G_e^{\prime}(\delta)&=&\theta(\delta)+\theta^{\prime}(\delta)\delta-\psi^{\prime}(\theta(\delta))\theta^{\prime}(\delta)\\
G_e^{\prime\prime}(\delta)&=&2\theta^{\prime}(\delta)+\theta^{\prime\prime}(\delta)\delta-\psi^{\prime\prime}(\theta(\delta))\theta^{\prime}(\delta)^2-\psi^{\prime}(\theta(\delta))\theta^{\prime\prime}(\delta)\\
G_e^{\prime\prime\prime}(\delta)&=& 3\theta^{\prime\prime}(\delta)+\theta^{\prime\prime\prime}(\delta)\delta-\psi^{\prime\prime\prime}(\theta(\delta))\theta^{\prime}(\delta)^3-2\psi^{\prime\prime}(\theta(\delta))\theta^{\prime}(\delta)\theta^{\prime\prime}(\delta)\\
&&-\psi^{\prime\prime}(\theta(\delta))\theta^{\prime\prime}(\delta)\theta^{\prime}(\delta)-\psi^{\prime}(\theta(\delta))\theta^{\prime\prime\prime}(\delta)
\end{eqnarray*}
We also notice that $\psi^{\prime}(\theta(\delta))=\delta$
Thus, 
$\theta^{\prime}(\delta)=\frac{1}{\psi^{\prime\prime}(\theta(\delta))} \mbox{ and }
\theta^{\prime\prime}(\delta)=-\frac{\psi^{\prime\prime\prime}(\theta(\delta))}{\psi^{\prime\prime}(\theta(\delta))^3}$.
The results then follows by noting that $\psi(0)=0$, and $\psi^{(k)}(0)=E[(X_2^0-X_1)^k]$ for $k=1,2,\dots$.
\end{proof}

\bibliographystyle{plain}
\bibliography{ordinal}

\end{document}